\documentclass{amsart}
\usepackage{amsmath,amsthm,amssymb,microtype,color,verbatim,hyperref,multicol,mathrsfs}
\usepackage{amsmath,amsthm,amssymb,microtype,color,verbatim,hyperref,multicol,tikz}
\usepackage{float}
\restylefloat{table}
\makeatletter
\def\imod#1{\allowbreak\mkern10mu({\operator@font mod}\,\,#1)}
\makeatother

\newtheorem{corollary}{Corollary}
\newtheorem{ex}{Example}
\newtheorem{theorem}{Theorem}
\newtheorem{pro}{Proposition}

\newtheorem{lm}{Lemma}

\title[Factorizations of Matrices With Recursive Entries and Related Topics]{Factorizations of Matrices With Recursive Entries and Related Topics}

\author[X. Y. Chen]{X. Y. Chen}
\address{School of Mathematics and Statistics, Henan University of Technology, China.}
\email{cxy19800222@163.com}

\author[A. R. Moghaddamfar]{A. R. Moghaddamfar}
\address{Faculty of Mathematics, K. N. Toosi University of Technology, Tehran, Iran.}
\email{moghadam@kntu.ac.ir}

\author[K. Moghaddamfar]{K. Moghaddamfar}
\address{Department of Mathematics, Federal University of Cear\'{a}, Fortaleza, Brazil.}
\email{kambiz.moghaddamfar@gmail.com}
\begin{document}
\maketitle
\begin{abstract}
This article examines matrices whose entries are determined by recursive relations of the form $A_{i, j} = x A_{i, j-1} + y A_{i-1, j-1} + z A_{i-1, j}$, where $x, y, z$ are constants, and the initial conditions are defined along the first row and column. We present a general decomposition for such matrices and show that many of the known decompositions are particular cases of this more general decomposition. Additionally, we provide a decomposition of these matrices into Pascal-like matrices and a basic Toeplitz matrix.
\end{abstract}

\footnotetext{{\em $2020$ Mathematics Subject Classification}: 15A15, 15A23, 11C20. }

\section{Introduction}
In this article, we focus on infinite matrices $A = [A_{i, j}]_{i, j \geqslant 0}$ whose entries satisfy the following recurrence relation:
\[A_{i, j} = x A_{i, j-1} + yA_{i-1,j-1} + zA_{i-1,j} \quad \text{for } i, j \geq 1;\]
where $x$, $y$, and $z$ are arbitrary constants.

Matrices of this type have been the focus of numerous studies, particularly regarding their factorization into simpler matrices and the evaluation of their determinants (see, for instance, \cite{K2, M, T}). In these studies, authors often present direct decompositions of these matrices using Pascal, Pascal-like, and Toeplitz matrices. Through these decompositions, they investigate the calculation of determinants.

This article aims to clarify these decompositions by utilizing the action of a generalized Pascal matrix on this family of matrices. In most instances, this approach either generalizes previous decompositions or provides an alternative method.

Let $x$, $y$, and $z$ be indeterminates, and let $\alpha=(\alpha_i)_{i\geq 0}$ and $\beta=(\beta_i)_{i\geq 0}$ be two sequences starting with a common first term $\alpha_0=\beta_0=\gamma$.  
Define an infinite matrix $P_{\alpha, \beta}^{[x, y, z]} = [P_{i, j}]_{i, j \geq 0}$ by the following recurrence relation and initial conditions:
\[P_{i, j} = xP_{i, j-1}+yP_{i-1, j-1}+zP_{i-1, j}, \  \text{for} \ i, j\geq 1,\]
and 
$$P_{i, 0}= \alpha_i  \  \text{and} \  P_{0, j}=\beta_j \ \text{for}   \ i, j\geq 0.$$
The matrix $P_{\alpha, \beta}^{[x, y, z]}$ is called the {\em weighted  recurrence matrix} with parameters $x$, $y$, and $z$, associated with $\alpha$ and $\beta$. 

Some specific cases are as follows:
\begin{itemize}
\item[(a)] The weighted recurrence matrix $P_{\alpha, \beta}^{[1, 0, 1]}$ is a {\em generalized Pascal triangle} (see \cite{Bach}). When $\alpha=(1, 1, 1, \ldots)$ and $\beta=(1, 1, 1, \ldots)$, the weighted recurrence matrix $P_{\alpha, \beta}^{[1, 0, 1]}$ is the {\em infinite classical Pascal matrix}.  
\item[(b)] The weighted recurrence matrix $P_{\alpha, \beta}^{[0, 1, 0]}$ is an infinite \emph{Toeplitz matrix} in the usual sense. A more general version of this concept is the \emph{weighted Toeplitz matrix} $P_{\alpha, \beta}^{[0, x, 0]}$ for an indeterminate $x$. It is represented as follows:
\[P_{\alpha, \beta}^{[0, x, 0]}=
\begin{pmatrix}
\gamma & \beta_1 &  \beta_2 &  \beta_3 & \cdots \\
\alpha_1 & \gamma x &  \beta_1 x &  \beta_2 x & \cdots \\
\alpha_2 & \alpha_1 x & \gamma x^2 &  \beta_1 x^2 & \cdots \\
\alpha_3 & \alpha_2 x &  \alpha_1 x^2 &  \gamma x^3 & \cdots \\
\vdots & \vdots & \vdots & \vdots & \ddots
\end{pmatrix}.
\]

\item[(c)] The weighted recurrence matrix $P_{\alpha, \beta}^{[0, 1, 1]}$ is a {\em $7_{\alpha, \beta}$-matrix} (see \cite{JZ}). A more general version of this concept is the \emph{weighted $7_{\alpha, \beta}$-matrix} $P_{\alpha, \beta}^{[0, y, z]}$ for indeterminates $y$ and $z$. Especially, we will concentrate on the weighted $7_{\alpha, \beta}$-matrices $P_{\alpha, \beta}^{[0, y, z]}$ where $\alpha=(z^i)_{i\geq 0}$ and  $\beta=(1, 0, 0, \ldots)$. It is represented as follows:
\[P_{\alpha, \beta}^{[0, y, z]}=
\begin{pmatrix}
1 & 0 & 0 & 0 & \cdots \\
z & y & 0 & 0 & \cdots \\
z^2 & 2yz & y^2 & 0 & \cdots \\
z^3 & 3yz^2 & 3y^2z & y^3& \cdots \\
\vdots & \vdots & \vdots & \vdots & \ddots
\end{pmatrix}=\left( \binom{i}{j} y^jz^{i-j}\right)_{i, j\geq 0}.
\]
\end{itemize}

Unless noted otherwise, we use the following notation in the rest of the paper:
\begin{itemize}
\item For $z\in {\mathbb C}$, $\lambda_z=(z^i)_{i\geq 0}$.
\item $\mu=(1, 0, 0, 0, \ldots)=\lambda_0$ (using the convention $0^0=1$). 
\item The symbol $A^t$ denotes the transpose of a matrix $A$.
\end{itemize}
For the sake of simplicity in notation, when we deal with an infinite matrix $A(\infty)$, we will simply use $A$ to represent the principal $n \times n$ submatrix, rather than using $A(n)$, as long as there is no possibility of confusion.

\section{Group Actions on Weighted  Recurrence Matrices} 
 The following lemma shows that for two complex numbers $v \neq 0$ and $w$, the weighted $7_{\lambda_w, \mu}$-matrices $P_{\lambda_w, \mu}^{[0, v, w]}$ form a group under matrix multiplication.
\begin{lm}\label{lm1} 
The collection $\mathsf {F}$ consists of $n\times n$ weighted  $7_{\lambda_w, \mu}$-matrices $P_{\lambda_w, \mu}^{[0, v, w]}$, where $v$ and $w$ are two complex numbers with $v\neq 0$, 
forms a group under matrix multiplication. Furthermore, if we define $\mathsf {F}^t=\{A^t \ | \ A\in \mathsf {F}\}$, then  $\mathsf {F}^t$ also forms a group under matrix multiplication.
\end{lm}
\begin{proof} Let $ v, w, y, z\in {\Bbb C}$ with $v, y\neq 0$. Then, an easy computation establishes the identity:
\begin{align}\label{eq-0} P_{\lambda_w, \mu}^{[0, v, w]}  P_{\lambda_z, \mu}^{[0, y, z]}=P_{\lambda_{vz+w}, \mu}^{[0, vy, vz+w]},\end{align}
(see also  (\cite[2.4]{T}). This shows that $\mathsf {F}$ is closed under matrix multiplication. Additionally, we see that $P_{\lambda_{0}, \mu}^{[0, 1, 0]}$ is the identity element,  and 
\[(P_{\lambda_w, \mu}^{[0, v, w]})^{-1}=P_{\lambda_{-wv^{-1}}, \mu}^{[0, v^{-1}, -wv^{-1}]}.\]
Thus, the set $\mathsf {F}$, under the matrix multiplication, forms a group. 

In addition, we have 
\begin{align*} (P_{\lambda_w, \mu}^{[0, v, w]})^t  (P_{\lambda_z, \mu}^{[0, y, z]})^t=(P_{\lambda_z, \mu}^{[0, y, z]}P_{\lambda_w, \mu}^{[0, v, w]} )^t=(P_{\lambda_{yw+z}, \mu}^{[0, yv, yw+z]})^t.\end{align*}
Reasoning as before, it is clear that $\mathsf {F}^t$ forms a group under matrix multiplication. 
\end{proof}

In the sequel, we examine two actions of the group $\mathsf {F}$ on the weighted recurrence matrices. We denote the family of $n \times n$ submatrices of the weighted recurrence matrices $P_{\alpha, \beta}^{[x, y, z]}$ as $\mathrm{WRM}(n)$.

\begin{pro}\label{pro1} The group $\mathsf {F}$ acts on $\mathrm{WRM}(n)$ by left multiplication.
\end{pro}
\begin{proof} Let $P_{\lambda_w, \mu}^{[0, v, w]}\in \mathsf {F}$ and $P_{\alpha, \beta}^{[x, y, z]}\in \mathrm{WRM}(n)$, where $\alpha, \beta$ are two arbitrary sequences with $\alpha_0=\beta_0$. We show first that 
\begin{align}\label{eq1}
P_{\lambda_w, \mu}^{[0, v, w]} \cdot P_{\alpha, \beta}^{[x, y, z]}=P^{[x, vy-wx, vz+w]}_{\eta, \beta},
\end{align}
where $\eta=(\eta_i)_{i\geq 0}$ with 
$$\eta_i= \sum_{k=0}^{i} \binom{i}{k} v^k w^{i-k} \alpha_{k},  \ \text{for} \ i\geq 0.$$
For simplicity, let $A=P_{\lambda_w, \mu}^{[0, v, w]}$, $B=P_{\alpha, \beta}^{[x, y, z]}$ and $C=P^{[x,  vy-wx, vz+w]}_{\eta, \beta}$. To establish Eq. (\ref{eq1}), we observe first that 
$$ (AB)_{i, 0}=\sum_{k=0}^{n-1} A_{i, k} B_{k, 0}=\sum_{k=0}^{i} {i\choose k} v^k w^{i-k} \alpha_k=\eta_i=C_{i, 0},$$
and 
$$(AB)_{0, j}=\sum_{k=0}^{n-1} A_{0, k} B_{k, j}= A_{0,0}B_{0, j}=1\cdot \beta_j= \beta_j=C_{0, j}.
$$
Next, assuming that $1\leq i, j\leq n-1$, we show that 
\begin{align}\label{eq2}
(AB)_{i, j}=x(AB)_{i, j-1}+(vy-wx) (AB)_{i-1, j-1}+(vz+w) (AB)_{i-1, j}.
\end{align}
To prove Eq. (\ref{eq2}), we begin simplifying the right-hand side of the equation.
First, since $A_{i, j}=vA_{i-1, j-1}+wA_{i-1, j}$,  
it follows that 
\begin{align}\label{eq3}
x(AB)_{i, j-1}=xA_{i,0}B_{0,j-1}+\sum_{k=1}^{n-1} x (v A_{i-1, k-1}+wA_{i-1, k})  B_{k, j-1}.  
\end{align}
Similarly, the second term on the right-hand side of Eq. (\ref{eq2}) equals 
\begin{align}\label{eq4}
 (vy-wx) (AB)_{i-1, j-1}=\sum_{k=0}^{n-1}  vy A_{i-1, k} B_{k, j-1}-\sum_{k=0}^{n-1} wx A_{i-1,k} B_{k, j-1}.
\end{align}
Adding Eqs. (\ref{eq3}) and (\ref{eq4}), and noting that  $A_{i, 0}=wA_{i-1,0}\), we conclude that
\begin{align*}
x(AB)_{i, j-1}+(vy-wx) (AB)_{i-1, j-1}= \sum_{k=1}^{n-1} vx A_{i-1,k-1} B_{k, j-1}+\sum_{k=0}^{n-1} vy A_{i-1, k} B_{k, j-1}.
\end{align*}
Now, applying the index shifting $k\rightarrow k+1$ and noting that $A_{i-1,n-1}=0$, we get
\begin{align*}
x(AB)_{i, j-1}+ (vy-wx) (AB)_{i-1, j-1}=\sum_{k=0}^{n-2} vx A_{i-1, k} B_{k+1,j-1}+\sum_{k=0}^{n-2} vy A_{i-1,k} B_{k, j-1},
\end{align*}
and since $B_{k, j}=xB_{k, j-1}+yB_{k-1, j-1}+zB_{k-1, j}$, we obtain
\begin{align}
x(AB)_{i, j-1}+ (vy-wx) (AB)_{i-1, j-1}&=\sum_{k=0}^{n-2}  vA_{i-1,k} (xB_{k+1,j-1}+yB_{k,j-1})\notag \\&=\sum_{k=0}^{n-2} vA_{i-1, k} (B_{k+1,j}-z B_{k, j}). \label{eq5}
\end{align}
Finally, using completely analogous arguments, the third term of the right-hand side of Eq. (\ref{eq2}) is equal to 
\begin{align}\label{eq6}
(vz+\omega)\sum_{k=0}^{n-1}  A_{i-1,k} B_{k, j}=\sum_{k=0}^{n-1} vz A_{i-1,k} B_{k, j}+\sum_{k=0}^{n-1} wA_{i-1,k} B_{k, j}. 
\end{align}
Adding Eqs. (\ref{eq5}) and (\ref{eq6}), the right-hand side of  Eq. (\ref{eq2}) equals 
\begin{align*}
& x(AB)_{i, j-1}+(vy-wx) (AB)_{i-1, j-1}+(vz+w) (AB)_{i-1, j}\\
&\ \ \ \ \ =\sum_{k=0}^{n-2} vA_{i-1,k} (B_{k+1, j}-zB_{k, j})+\sum_{k=0}^{n-1} vz A_{i-1,k} B_{k, j}+\sum_{k=0}^{n-1} wA_{i-1,k} B_{k, j}\\
&\ \ \ \ \ =\sum_{k=0}^{n-2}  vA_{i-1,k} B_{k+1, j}+\sum_{k=0}^{n-1} wA_{i-1,k} B_{k, j}\\
&\ \ \ \ \ =\sum_{k=1}^{n-1} vA_{i-1, k-1} B_{k, j}+\sum_{k=0}^{n-1} wA_{i-1, k} B_{k, j} \ \ \text{(the index shifting $k\rightarrow k-1$)}\\
&\ \ \ \ \ = \sum_{k=1}^{n-1}  \big(vA_{i-1,k-1}+ w A_{i-1, k}\big) B_{k, j}+A_{i-1, 0}A_{0, j}\\
&\ \ \ \ \ = \sum_{k=1}^{n-1}  A_{i, k} B_{k, j}+A_{i, 0}A_{0, j}\\
&\ \ \ \ \ = \sum_{k=0}^{n-1}  A_{i, k} B_{k, j}\\
&\ \ \ \ \ = (AB)_{i, j},
\end{align*}
which corresponds to the left-hand side of  Eq. (\ref{eq2}). This completes the proof of Eq. (\ref{eq1}).

With Eq. (\ref{eq1}) in mind, we can now define an action of $\mathsf {F}$ on $\mathrm{WRM}(n)$ by setting 
\begin{align*}(P_{\alpha, \beta}^{[x, y, z]})^{P_{\lambda_w, \mu}^{[0, v, w]}}&=(P_{\lambda_w, \mu}^{[0, v, w]})^{-1} P_{\alpha, \beta}^{[x, y, z]}\\
&=P_{\lambda_{-wv^{-1}}, \mu}^{[0, v^{-1}, -wv^{-1}]} P_{\alpha, \beta}^{[x, y, z]}\\ &=P^{[x,  v^{-1}y+wv^{-1}x, v^{-1}z-wv^{-1}]}_{\eta, \beta}.\end{align*}
To verify that this really is an action, let $P_{\lambda_w, \mu}^{[0, v, w]}, P_{\lambda_s, \mu}^{[0, r, s]}\in \mathsf {F}$. Then for $P_{\alpha, \beta}^{[x, y, z]}\in \mathrm{WRM}(n)$, we have 
$$(P_{\alpha, \beta}^{[x, y, z]})^{P_{\lambda_{0}, {\mu}}^{[0, 1, 0]}}=(P_{\lambda_{0}, {\mu}}^{[0, 1, 0]} )^{-1}P_{\alpha, \beta}^{[x, y, z]}=P_{\lambda_{0}, {\mu}}^{[0, 1, 0]} P_{\alpha, \beta}^{[x, y, z]}=P^{[x,  y, z]}_{\alpha, \beta},$$ and also 
\begin{align*} (P_{\alpha, \beta}^{[x, y, z]})^{P_{\lambda_w, \mu}^{[0, v, w]}P_{\lambda_s, \mu}^{[0, r, s]}}&=(P_{\lambda_w, \mu}^{[0, v, w]}P_{\lambda_s, \mu}^{[0, r, s]})^{-1} P_{\alpha, \beta}^{[x, y, z]}\\ 
&=(P_{\lambda_s, \mu}^{[0, r, s]})^{-1}  (P_{\lambda_w, \mu}^{[0, v, w]})^{-1}P^{[x,  y, z]}_{\alpha, \beta}\\
&=(P_{\lambda_s, \mu}^{[0, r, s]})^{-1} (P_{\alpha, \beta}^{[x, y, z]})^{P_{\lambda_w, \mu}^{[0, v, w]}}\\
&=\left((P_{\alpha, \beta}^{[x, y, z]})^{P_{\lambda_w, \mu}^{[0, v, w]}}\right)^{P_{\lambda_s, \mu}^{[0, r, s]}}.
\end{align*}
Hence, the proof is complete. \end{proof}

Similarly, an action of $\mathsf {F}^t$ can be defined on $\mathrm{WRM}(n)$ through right multiplication. Therefore, we have the following proposition:

\begin{pro}\label{pro2} The group $\mathsf {F}^t$ acts on $\mathrm{WRM}(n)$ by right multiplication.
\end{pro}
\begin{proof} Reasoning as we did previously at Proposition \ref{pro1}, assuming 
$P_{\lambda_w, \mu}^{[0, v, w]}\in \mathsf {F}$ and $P_{\alpha, \beta}^{[x, y, z]}\in \mathrm{WRM}(n)$, where $\alpha, \beta$ are two arbitrary sequences with $\alpha_0=\beta_0$, we have  
\begin{align}\label{eq7}
P_{\alpha, \beta}^{[x, y, z]}\cdot (P_{\lambda_w, \mu}^{[0, v, w]})^t=P^{[vx+w, vy-wz, z]}_{\alpha, \eta},
\end{align}
where $\eta=(\eta_i)_{i\geq 0}$ with 
$$\eta_i= \sum_{k=0}^{i} \binom{i}{k} v^k  w^{i-k} \beta_{k},  \ \text{for} \  i\geq 0.$$
We can now define an action of $\mathsf {F}^t$ on $\mathrm{WRM}(n)$ by setting 
\begin{align*}(P_{\alpha, \beta}^{[x, y, z]})^{(P_{\lambda_w, \mu}^{[0, v, w]})^t}=P_{\alpha, \beta}^{[x, y, z]} (P_{\lambda_w, \mu}^{[0, v, w]})^{t}=P^{[vx+w, vy-wz, z]}_{\alpha, \eta}.\end{align*}
To complete the proof, it suffices to show that this is an action. To see this, let $P_{\lambda_w, \mu}^{[0, v, w]}, P_{\lambda_s, \mu}^{[0, r, s]}\in \mathsf {F}$. Then for $P_{\alpha, \beta}^{[x, y, z]}\in \mathrm{WRM}(n)$, we have 
$$(P_{\alpha, \beta}^{[x, y, z]})^{(P_{\lambda_{0}, {\mu}}^{[0, 1, 0]})^t}=P_{\alpha, \beta}^{[x, y, z]}(P_{\lambda_{0}, {\mu}}^{[0, 1, 0]} )^{t}=P_{\alpha, \beta}^{[x, y, z]}P_{\lambda_{0}, {\mu}}^{[0, 1, 0]} =P^{[x,  y, z]}_{\alpha, \beta},$$ and also using Eq. (\ref{eq-0}), we obtain
\begin{align*} (P_{\alpha, \beta}^{[x, y, z]})^{(P_{\lambda_w, \mu}^{[0, v, w]})^t (P_{\lambda_s, \mu}^{[0, r, s]})^t}&=(P_{\alpha, \beta}^{[x, y, z]})^{(P_{\lambda_s, \mu}^{[0, r, s]} P_{\lambda_w, \mu}^{[0, v, w]})^t}\\
& =(P_{\alpha, \beta}^{[x, y, z]})^{(P_{\lambda_{rw+s}, \mu}^{[0, rv, rw+s]})^t}\\ 
&= P_{\alpha, \beta}^{[x, y, z]} (P_{\lambda_{rw+s}, \mu}^{[0, rv, rw+s]})^t\\
&= (P_{\alpha, \beta}^{[x, y, z]}(P_{\lambda_s, \mu}^{[0, r, s]} P_{\lambda_w, \mu}^{[0, v, w]})^t\\
&= P_{\alpha, \beta}^{[x, y, z]} (P_{\lambda_w, \mu}^{[0, v, w]})^t (P_{\lambda_s, \mu}^{[0, r, s]})^t \\
&=\left((P_{\alpha, \beta}^{[x, y, z]})^{(P_{\lambda_w, \mu}^{[0, v, w]})^t}\right)^{(P_{\lambda_s, \mu}^{[0, r, s]})^t}.
\end{align*}
This completes the proof. \end{proof}

\section{A Factorization of Matrices With Recursive Entries}

By combining Propositions \ref{pro1} and \ref{pro2}, we derive the following theorem.

\begin{theorem}[Unifying Factorization]\label{mainth}
For any $r, s, v, w, x, y, z\in \mathbb{C}\) such that $rv\neq 0$, and any two sequences $\alpha, \beta$ with $\alpha_0=\beta_0$, we have the following decomposition:
\[
P_{\alpha, \beta}^{[x, y, z]}=P_{\lambda_s, \mu}^{[0, r, s]}  \cdot P_{\rho, \sigma}^{\left[ \frac{x-w}{v}, \frac{y+xs+zw-sw}{rv}, \frac{z-s}{r} \right]}\cdot (P_{\lambda_w, \mu}^{[0, v, w]} )^t,
\]
where the transformed sequences $\rho=(\rho_i)_{i\geq 0}$ and  $\sigma=(\sigma_i)_{i\geq 0}$ are given by:
\begin{align*}
\rho_{0}=\alpha_0, \ \text{and} \  \rho_{i}= \left(\alpha_i - \sum_{k=0}^{i-1} \binom{i}{k} r^k s^{i-k}  \rho_k\right)/r^i,  \quad \text{for } \  i\geq 1,\\
\sigma_{0}=\beta_0, \ \text{and} \  \sigma_{i}=\left( \beta_i-\sum_{k=0}^{i-1} \binom{i}{k} v^k w^{i-k}  \sigma_k\right)/v^i,  \quad \text{for } \   i\geq 1.
 \end{align*}
\end{theorem}
\begin{proof}
Applying Eq. (\ref{eq1}), we obtain:
\begin{align*}
P_{\lambda_s, \mu}^{[0, r, s]}  \cdot P_{\rho, \sigma}^{\left[ \frac{x-w}{v}, \frac{y+xs+zw-sw}{rv}, \frac{z-s}{r} \right]} =P_{\eta, \sigma}^{[\frac{x-w}{v}, \frac{y+zw}{v}, z]}
\end{align*}
where $\eta=(\eta_i)_{i\geq 0}$ satisfies the following relation:
\begin{align*}
\eta_{i} =\sum_{k=0}^{i} \binom{i}{k} r^k s^{i-k} \rho_{k}, \quad   \text{for}  \   i\geq 0.
\end{align*}
Note that $\eta_i=\alpha_i$ for all $i \geq 0$. In fact, we can demonstrate that:
\begin{align*}
\eta_i&=\sum_{k=0}^{i} \binom{i}{k} r^k s^{i-k} \rho_{k}=r^i \rho_{i}+\sum_{k=0}^{i-1}\binom{i}{k} r^k s^{i-k} \rho_{k}\\ &= \left(\alpha_i-\sum_{k=0}^{i-1} \binom{i}{k} r^k s^{i-k}  \rho_k\right)+\sum_{k=0}^{i-1}\binom{i}{k} r^k s^{i-k} \rho_{k}=\alpha_i.
\end{align*}
Thus, we can express the above equation as:
\begin{align*}
P_{\lambda_s, \mu}^{[0, r, s]}  \cdot P_{\rho, \sigma}^{\left[ \frac{x-w}{v}, \frac{y+xs+zw-sw}{rv}, \frac{z-s}{r} \right]} =P_{\alpha, \sigma}^{[\frac{x-w}{v}, \frac{y+zw}{v}, z]}.
\end{align*}
Applying Eq. (\ref{eq7}) to $P_{\alpha, \sigma}^{[\frac{x-w}{v}, \frac{y+zw}{v}, z]}$, we have:
\begin{align*}
P_{\alpha, \sigma}^{[\frac{x-w}{v}, \frac{y+zw}{v}, z]} \cdot (P_{\lambda_w, \mu}^{[0, v, w]})^t=P_{\alpha, \nu}^{[x, y, z]},
\end{align*}
where $\nu=(\nu_i)_{i\geq 0}$ satisfies the following relation:
\begin{align*}
\nu_i= \sum_{k=0}^{i} \binom{i}{k} v^k w^{i-k} \sigma_{k}, \ \ 
\text{for} \ i\geq 0. 
\end{align*}
 Reasoning as before, we can observe that $\nu_i = \beta_i$ for all $i \geq 0$,  which completes the proof.
\end{proof}

The applicability of Theorem \ref{mainth} enables us to recover and extend many established results. 
For example, the main theorem presented in \cite[Theorem 1]{MP} demonstrates that for any arbitrary sequences $\alpha$ and $\beta$ with $\alpha_0= \beta_0$, it follows that 
\begin{align*}
P_{\alpha, \beta}^{[1, 0, 1]}=P_{\lambda_1, \mu}^{[0, 1, 1]}  P_{\hat{\alpha}, \hat{\beta}}^{[0, 1, 0]} 
(P_{\lambda_1, \mu}^{[0, 1, 1]})^t
\end{align*} 
where $\hat{\alpha}=(\hat{\alpha}_{i})_{i\geq 0}$ and
$\hat{\beta}=(\hat{\beta}_{i})_{i\geq 0}$ are two sequences with
$$\hat{\alpha}_{i}=\sum_{k=0}^{i}(-1)^{i+k}{i\choose k}\alpha_k \ \ \
\mbox{and} \ \ \
\hat{\beta}_{i}=\sum_{k=0}^{i}(-1)^{i+k}{i\choose k}\beta_k.$$
In particular, if $\alpha=(1, 1, 1, \ldots)$ and $\beta=(1, 1, 1, \ldots)$, then we have   
$\hat{\alpha}=(1, 0, 0, \ldots)$ and  $\hat{\beta}=(1, 0, 0, \ldots)$, and so $P_{\hat{\alpha}, \hat{\beta}}^{[0, 1, 0]}=P_{\lambda_0, \mu}^{[0, 1, 0]} =I$ (the identity matrix), hence we obtain
\begin{align*}
P_{\alpha, \beta}^{[1, 0, 1]}=P_{\lambda_1, \mu}^{[0, 1, 1]} 
(P_{\lambda_1, \mu}^{[0, 1, 1]})^t,
\end{align*} 
(see \cite{ES}). However, the above results are special cases of Theorem \ref{mainth}. In fact, in Theorem \ref{mainth},  if we take $r=s=v=w=1$ and if $y$ is an arbitrary complex number, then we have 
\begin{align*}
P_{\alpha, \beta}^{[1, y, 1]}=P_{\lambda_1, \mu}^{[0, 1, 1]}  P_{\rho, \sigma}^{[0, y+1, 0]} 
(P_{\lambda_1, \mu}^{[0, 1, 1]})^t.
\end{align*} 
Note that, in this situation, we have $\rho=\hat{\alpha}$ and $\sigma=\hat{\beta}$, and so 
\begin{align*}
P_{\alpha, \beta}^{[1, y, 1]}=P_{\lambda_1, \mu}^{[0, 1, 1]}  P_{\hat{\alpha}, \hat{\beta}}^{[0, y+1, 0]} 
(P_{\lambda_1, \mu}^{[0, 1, 1]})^t.
\end{align*} 
One the other hand, one of the fundamental decomposition theorems for the weighted recurrence matrices was first established by Tan Mingshu~\cite{T}, who proved that when $xz\neq0$, the matrix $P_{\alpha, \beta}^{[x, y, z]}$ admits a factorization of the form 
\[
P_{\alpha, \beta}^{[x, y, z]}=P_{\lambda_z, \mu}^{[0, z, z]} P_{\rho, \sigma}^{[0, 1+y/(xz), 0]} (P_{\lambda_x, \mu}^{[0, x, x]} )^t.
\]
Here one can also see that this decomposition is a special case of Theorem \ref{mainth} when we take $r=s=z$ and $v=w=x$. In fact, we have:
\begin{align*}
P_{\alpha, \beta}^{[x, y, z]}=P_{\lambda_z, \mu}^{[0, z, z]} P_{\rho, \sigma}^{[0, 1+y/(xz), 0]} (P_{\lambda_x, \mu}^{[0, x, x]})^t.
\end{align*}
which is the same result as obtained by Tan Mingshu~\cite{T}.

Tan's decomposition utilizes a weighted Toeplitz matrix. A significant result from Propositions \ref{pro1} and \ref{pro2} is that we can also decompose $P_{\alpha, \beta}^{[x, y, z]}$ into a product that involves a \emph{simple} (unweighted) Toeplitz matrix. This provides us with a more straightforward way to express the decomposition.
\begin{corollary} \label{C1}
For any matrix $P_{\alpha, \beta}^{[x, y, z]}$, such that $y+xz\neq 0$ we have the following decomposition:
\begin{align}\label{eq8}
P_{\alpha, \beta}^{[x, y, z]}=P_{\lambda_z, \mu}^{[0, 1, z]}\cdot  P_{\widetilde{\alpha}, \widetilde{\beta}}^{[0, 1, 0]} \cdot (P_{\lambda_x, \mu}^{[0, y+xz, x]})^t,
\end{align}
where the transformed sequences $\widetilde{\alpha}$, $\widetilde{\beta}$ are given by:
\begin{align}\label{eq9} \widetilde{\alpha}_{0}=\alpha_0, \ \text{and} \ \  \widetilde{\alpha}_{i}=\alpha_i-\sum_{k=0}^{i-1}\binom{i}{k} z^{i-k} \widetilde{\alpha}_k, \   i\geq 1,
\end{align}
and 
 \begin{align}\label{eq10} \widetilde{\beta}_{0}=\beta_0, \  \text{and} \ \  \widetilde{\beta}_{i}=\left(\beta_i - \sum_{k=0}^{i-1} \binom{i}{k} x^{i-k} (y+xz)^k \widetilde{\beta}_k\right)/(y+xz)^i, \  i\geq 1.\end{align}
\end{corollary}
\begin{proof}
The proof follows a similar approach to Theorem \ref{mainth}. We need just to take $\widetilde{\alpha}$ and $\widetilde{\beta}$ as given in the statement and  using Propositions \ref{pro1} and  \ref{pro2} to compute the right-hand side of Eq. (\ref{eq8}).
\end{proof}

\begin{ex} {\rm  Given $a, b\in {\Bbb C}$ with $ab\neq 0$, we have 
$$P_{\lambda_a, \lambda_b}^{[b, ab, a]}=P_{\lambda_a, \mu}^{[0, 1, a]}\cdot  P_{\widetilde{\alpha}, \widetilde{\beta}}^{[0, 1, 0]} \cdot (P_{\lambda_b, \mu}^{[0, 2ab, b]})^t.$$
Here, of course, $\widetilde{\alpha}=\widetilde{\beta}=(1, 0, 0, \ldots)$ and so $P_{\widetilde{\alpha}, \widetilde{\beta}}^{[0, 1, 0]}=I$.
In particular, for $n=3$, we have the following decomposition:
\[
\begin{pmatrix}
1 & b & b^2  \\
a & 3ab & 5ab^2 \\
a^2 & 5a^2b & 13a^2b^2\\
\end{pmatrix}
=
\begin{pmatrix}
1 & 0 & 0 \\
a & 1 & 0 \\
a^2 &2a & 1 \\

\end{pmatrix}
\cdot
\begin{pmatrix}
1 & 0 & 0  \\
0 & 1 & 0  \\
0 & 0 & 1 \\
\end{pmatrix}
\cdot
\begin{pmatrix}
1 & 0 & 0 \\
b & 2ab & 0 \\
b^2 & 4ab^2 & 4a^2b^2  \\
\end{pmatrix}^t.
\]
Furthermore, we obtain $\det P_{\lambda_a, \lambda_b}^{[b, ab, a]}=(2ab)^{\binom{n}{2}}$ (see also Corollary \ref{cor2}).}
\end{ex}

\begin{ex} {\rm  For $a\in {\Bbb C}\setminus \{0\}$, we have 
$$P_{\lambda_{a+1}, \lambda_{a+1}}^{[a, 1-a^2, a]}=P_{\lambda_a, \mu}^{[0, 1, a]}\cdot  P_{\widetilde{\alpha}, \widetilde{\beta}}^{[0, 1, 0]} \cdot (P_{\lambda_a, \mu}^{[0, 1, a]})^t.$$
Here, $\widetilde{\alpha}=\widetilde{\beta}=(1, 1, 1, \ldots)$ and so $P_{\widetilde{\alpha}, \widetilde{\beta}}^{[0, 1, 0]}=J$ (the unit matrix).
In particular, for $n=3$, we have the following decomposition:
\begin{align*}
&\begin{pmatrix}
1 & a+1 & (a+1)^2\\
a+1 & (a+1)^2 & (a+1)^3 \\
(a+1)^2 & (a+1)^3 & (a+1)^4\\
\end{pmatrix}=
\begin{pmatrix}
1 & 0 & 0 \\
a & 1 & 0 \\
a^2 &2a & 1 \\
\end{pmatrix}
\cdot
\begin{pmatrix}
1 & 1 & 1  \\
1 & 1 & 1  \\
1 & 1 & 1 \\
\end{pmatrix}
\cdot
\begin{pmatrix}
1 & 0 & 0 \\
a & 1 & 0 \\
a^2 & 2a & 1  \\
\end{pmatrix}^t.
\end{align*}}
\end{ex}

\section{Evaluations of Some Determinants of Matrices in $\mathrm{WRM}(n)$} The study of determinants for specific subsets of weighted recurrence matrices $\mathrm{WRM}(n)$ has been explored in various literature sources (see, for example, \cite{Bach, K2, MSS1}). Our decomposition of the elements of $\mathrm{WRM}(n)$, in accordance with Corollary \ref{C1}, provides a powerful tool for evaluating determinants. The key advantages are:
\begin{itemize}
\item[-] The determinants of elements in the group $\mathsf{F}$ are straightforward to compute.
\item[-]  Toeplitz determinants have been extensively studied and often feature recursive formulas. 
\end{itemize}

According to the decomposition in Corollary \ref{C1}, we have 
\[P_{\alpha,\beta}^{[x, y, z]} = P_{\lambda_z, \mu}^{[0, 1, z]} \cdot P_{\widetilde{\alpha},\widetilde{\beta}}^{[0,1,0]} \cdot (P_{\lambda_x,\mu}^{[0, y+xz, x]})^t,\]
and so
\[\det P_{\alpha,\beta}^{[x, y, z]}=\det P_{\lambda_z,\mu}^{[0,1,z]} \cdot \det P_{\widetilde{\alpha},\widetilde{\beta}}^{[0,1,0]} \cdot \det P_{\lambda_x,\mu}^{[0,y+xz, x]})^t.\]
Also, since 
\[\det P_{\lambda_a, \mu}^{[0, b, a]}=b^{\binom{n}{2}},\]
it follows that 
 \begin{align}\label{eq11}
 \det P_{\alpha,\beta}^{[x, y, z]}=(y+xz)^{\binom{n}{2}} \det P_{\widetilde{\alpha},\widetilde{\beta}}^{[0,1,0]}.
  \end{align}
Hence, the evaluation of determinants for matrices in $\mathrm{WRM}(n)$ can be simplified by computing the determinants of the associated Toeplitz matrices through our decomposition method. A particularly straightforward case arises when the central Toeplitz matrix becomes a diagonal matrix. However, achieving this condition imposes very restrictive requirements on the sequences $\alpha$ and $\beta$, as demonstrated in the following lemma.
\begin{pro}\label{pro3}
For $P_{\alpha, \beta}^{[x, y, z]} \in \mathrm{WRM}(n)$, the corresponding Toeplitz matrix $P_{\widetilde{\alpha}, \widetilde{\beta}}^{[0,1,0]}$ is a diagonal matrix  if and only if $\alpha=(\alpha_i)_{i\geq 0}$ and $\beta=(\beta_j)_{j\geq 0}$ are geometric sequences of the form $\alpha_i=c z^i$ and $\beta_j=c x^j$ for some constant $c\in \mathbb{C}$.
\end{pro}
\begin{proof}
If $P_{\widetilde{\alpha}, \widetilde{\beta}}^{[0,1,0]}$ is a diagonal matrix, then $\widetilde{\alpha}_0=\widetilde{\beta}_0=\alpha_0$ and $\widetilde{\alpha}_i=0=\widetilde{\beta}_j$ for all $i, j>1$. It is easy to see from Eqs. (\ref{eq9}) and (\ref{eq10}) that $\alpha_i= \alpha_0 z^i$ and $\beta_j=\alpha_0 x^j$ for all $i, j>1$. Now, the result follows by taking $c=\alpha_0$.

Conversely, suppose that $\alpha_i=c z^i$ and $\beta_j=c x^j$ for all $i, j\geq 0$. It is enough to show that 
$\widetilde{\alpha}=(\widetilde{\alpha}_i)_{i\geq 0}=(c, 0, 0,  \ldots)$ and $\widetilde{\beta}=(\widetilde{\beta}_j)_{j\geq 0}=(c, 0, 0,  \ldots)$. First of all, it is clear that $\widetilde{\alpha}_0=c=\widetilde{\beta}_0$. We proceed by induction on $i$ (resp. $j$) to show that $\widetilde{\alpha}_i=0$ (resp. $\widetilde{\beta}_j=0$) for all $i>0$ (resp. $j>0$). First, if $i=1$, then we obtain  
$$\widetilde{\alpha}_1=\alpha_1-z\widetilde{\alpha}_0=cz-cz=0.$$ 
Next, suppose that $i>1$. Now, it follows by the inductive hypothesis that
$$\widetilde{\alpha}_i=\alpha_i-\sum_{k=0}^{i-1} \binom{i}{k} z^{i-k} \widetilde{\alpha}_k=\alpha_i-\widetilde{\alpha}_0z^i=cz^i-cz^i=0.
$$
A similar argument shows that $\widetilde{\beta}=(c, 0, 0,  \ldots)$. 
Thus, $P_{\widetilde{\alpha},\widetilde{\beta}}^{[0,1,0]}=cI$ is a diagonal matrix.
\end{proof}

\begin{corollary}\label{cor2} If $\alpha=(\alpha_i)_{i\geq 0}$ and $\beta=(\beta_j)_{j\geq 0}$ are two geometric sequences of the form $\alpha_i=cz^i$ and $\beta_j=cx^j$ for some constant $c\in \mathbb{C}$, then we have
\[
\det P_{\alpha,\beta}^{[x, y, z]}=c^n(y+xz)^{\binom{n}{2}}.
\]
\end{corollary} 

Some results can be regarded as special cases of Corollary \ref{cor2}. In the sequel, we will consider three special cases:
\begin{itemize}
\item[-] If $c=x=z=1$ and $y=0$, then $P_{\alpha,\beta}^{[1, 0, 1]}$ is the classical Pascal matrix and we have $\det P_{\alpha,\beta}^{[1, 0, 1]}=1$.
\item[-] If $c=y=z=1$ and $x=0$, then $P_{\alpha,\beta}^{[0, 1, 1]}$ is an unipotent lower triangular matrix and we have 
$\det P_{\alpha,\beta}^{[0, 1, 1]}=1$.
\item[-] If $c=x=z=1$, then $\det P_{\alpha,\beta}^{[1, 1, 1]}=(y+1)^{\binom{n}{2}}$ (see also \cite[Theorem 1]{K2}).
\end{itemize}

 \begin{corollary}\cite[Theorem 1]{K2}\label{cor3} If $\alpha=(\alpha_i)_{i\geq 0}$ and $\beta=(\beta_i)_{i\geq 0}$ are two geometric sequences of the form $\alpha_i=a^i$ and $\beta_i=b^i$ for some constants $a, b\in \mathbb{C}$ and $y\in \mathbb{C}\setminus \{-1\}$, then we have
\[
\det P_{\alpha,\beta}^{[1, y, 1]}=(1+y)^{\binom{n-1}{2}}(y+a+b-ab)^{n-1}.
\]
\end{corollary} 
\begin{proof} By Eq. (\ref{eq11}), we have 
 \begin{align}\label{eq12} \det P_{\alpha,\beta}^{[1, y, 1]}=(1+y)^{\binom{n}{2}} \det P_{\widetilde{\alpha}, \widetilde{\beta}}^{[0,1,0]}.
 \end{align}
Now we compute $\widetilde{\alpha}$ and $\widetilde{\beta}$. It follows by Eqs. (\ref{eq9}) and (\ref{eq10}) that 
\begin{align*} \widetilde{\alpha}_{0}=1, \ \text{and} \ \  \widetilde{\alpha}_{i}=a^i-\sum_{k=0}^{i-1}\binom{i}{k} \widetilde{\alpha}_k, \   i\geq 1,
\end{align*}
and 
 \begin{align*} \widetilde{\beta}_{0}=1, \  \text{and} \ \  \widetilde{\beta}_{i}=\left(b^i-\sum_{k=0}^{i-1} \binom{i}{k}(1+y)^k \widetilde{\beta}_k\right)/(1+y)^i, \  i\geq 1.\end{align*}
 Trying induction on $i$, we see that $\widetilde{\alpha}_{i}=(a-1)^i$ and $ \widetilde{\beta}_i=\left(\frac{b-1}{1+y}\right)^i$ for all $i\geq 0$, and it is easy to see that 
$$\det P_{\widetilde{\alpha}, \widetilde{\beta}}^{[0,1,0]}=\left(1-(a-1)\left(\frac{b-1}{1+y}\right)\right)^{n-1}=\left(\frac{y+a+b-ab}{1+y}\right)^{n-1}.$$
The result now follows by substituting this into Eq. (\ref{eq12}). \end{proof}
 
\begin{corollary}\cite[Theorem 3]{K2}\label{cor4} If $\alpha=(\alpha_i)_{i\geq 0}$ and $\beta=(\beta_i)_{i\geq 0}$ are two arithmatic sequences of the form $\alpha_i=i$ and $\beta_i=-i$ and $y\in \mathbb{C}\setminus \{-1\}$, then we have
\[
\det P_{\alpha,\beta}^{[1, y, 1]}(2n)=(1+y)^{2n(n-1)}.
\]
\end{corollary} 
\begin{proof} Once again by Eq. (\ref{eq11}), we have 
 \begin{align}\label{eq13} \det P_{\alpha,\beta}^{[1, y, 1]}(2n)=(1+y)^{\binom{2n}{2}} \det P_{\widetilde{\alpha}, \widetilde{\beta}}^{[0,1,0]}(2n).
 \end{align}
 To calculate $\det P_{\widetilde{\alpha}, \widetilde{\beta}}^{[0,1,0]}(2n)$, we first need to determine the sequences $\widetilde{\alpha}\) and \(\widetilde{\beta}$. First, using Eqs. (\ref{eq9}) and (\ref{eq10}), we have 
\begin{align*} \widetilde{\alpha}_{0}=0, \ \text{and} \ \  \widetilde{\alpha}_{i}=i-\sum_{k=0}^{i-1}\binom{i}{k} \widetilde{\alpha}_k, \   i\geq 1,
\end{align*}
and 
 \begin{align*} \widetilde{\beta}_{0}=0, \  \text{and} \ \  \widetilde{\beta}_{i}=\left(-i-\sum_{k=0}^{i-1} \binom{i}{k}(1+y)^k \widetilde{\beta}_k\right)/(1+y)^i, \  i\geq 1.\end{align*}
 Trying induction on $i$, we see that $\widetilde{\alpha}_{i}=\delta_{i, 1}$ and $ \widetilde{\beta}_i=\frac{-\delta_{i, 1}}{(1+y)^i}$ for all $i\geq 0$, where the ``Kronecker delta" $\delta_{i, j}$ is $1$ if $i=j$ and is $0$ otherwise. It is now easy to show that
$$\det P_{\widetilde{\alpha}, \widetilde{\beta}}^{[0,1,0]}(2n)=\frac{1}{(1+y)^n}.$$
The result is obtained by substituting this into Equation (\ref{eq13}) and performing further simplifications.
\end{proof}

\section*{Acknowledgments} The first author thanks support of the program of Henan
University of Technology (2024PYJH019), the projects of Education Department of
Henan Province (YJS2022JC16, 25A110006), and the Natural Science Foundation of Henan Province
(242300421384). The third author would like to express his sincere gratitude to the
Department of Mathematics at the Federal University of Cear\'{a} and the project
\emph{Jangada Din\^{a}mica}, supported by the Instituto Serrapilheira
(\url{https://serrapilheira.org/en/}), for offering a postdoctoral position through the PGMAT-UFC.


\begin{thebibliography}{99}
\bibitem{Bach}  R. Bacher,  Determinants of matrices related to the Pascal triangle, {\em J. Th\'eor. Nombres Bordeaux},  {\bf 14}(1)(2002), 19--41.

\bibitem{ES} A. Edelman and G. Strong, Pascal matrices, {\em Amer. Math. Monthly}, {\bf 111}(3)(2004), 189--197.

\bibitem{K2} C. Krattenthaler, Evaluations of some determinants of matrices related to the Pascal triangle, {\em S\'em. Lothar. Combin.},  {\bf 47}(2001/02), Art. B47g, 19 pp.

\bibitem{M} A. R. Moghaddamfar, Determinants of several matrices associated with Pascal's triangle, {\em Asian-Eur. J. Math.}, {\bf 3}(1)(2010), 119--131.

\bibitem{MP} A. R. Moghaddamfar and S. M. H. Pooya,  Generalized Pascal triangles and Toeplitz matrices, {\em  Electron. J. Linear Algebra}, {\bf 18}(2009), 564--588.

\bibitem{MSS1} A. R. Moghaddamfar, S. Navid Salehy and S. Nima Salehy,  The determinants of matrices with recursive entries, {\em Linear Algebra Appl.},  428(11-12) (2008), 2468--2481.

\bibitem{T} M. Tan, Matrices associated to biindexed linear recurrence relations, {\em  Ars Combin.}, {\bf 86}(2008), 305--319.

\bibitem{JZ}  J. Zhong, Structural and sparsity properties of symmetric $7$-matrices,  {\em Linear Algebra Appl.}, {\bf 436}(2012),  2425--2439.

\end{thebibliography}
\end{document}